\documentclass[12pt]{article}
\usepackage{CJK}
\usepackage{latexsym}
\usepackage{amssymb, graphicx, amsmath}
\usepackage{booktabs}
\usepackage{cases}
\usepackage{color}
\definecolor{myblue}{rgb}{0,0,0.5}
\definecolor{mygreen}{rgb}{0,0.5,0}
\definecolor{myred}{rgb}{0.5,0,0}

\usepackage[normalem]{ulem}
\usepackage{enumerate}

\usepackage{amsmath,amsthm,amssymb}
\usepackage{graphicx}
\usepackage{subfigure}
\usepackage{caption}
\usepackage{epstopdf}
\usepackage{multirow}
\usepackage{url}
\usepackage{hyperref}

\usepackage{listings}
\usepackage{float}

\usepackage{threeparttable}

\newcommand{\nn}{\nonumber}
\newcommand{\RNum}[1]{\uppercase\expandafter{\romannumeral #1\relax}}

\def \[{\begin{equation}}
\def \]{\end{equation}}

\newcommand{\R}{\mathbf{R}}

\newtheorem{theorem}{Theorem}[section]

\newtheorem{lemma}{Lemma}[section]

\newtheorem{remark}{Remark}

\newtheorem{proposition}{Proposition}[section]

\begin{CJK}{GBK}{song}
\gdef\cdefinition{定义\,}
\gdef\clemma{引理\,}
\gdef\ctheorem{定理\,}
\end{CJK}

\begin{document}
\begin{CJK*}{GBK}{song}

\begin{center}

{\Large \bf A dual-primal balanced augmented Lagrangian method for linearly constrained convex programming}\\

\bigskip

 {\bf Shengjie Xu}\footnote{\parbox[t]{16cm}{
 Department of Mathematics, Harbin Institute of Technology, Harbin, China, and Department of Mathematics,   Southern  University of Science and Technology, Shenzhen, China.  Email: xsjnsu@163.com
  }}

\medskip

\today

\end{center}

\medskip

{\small

\parbox{0.95\hsize}{

\hrule

\medskip

{\bf Abstract.} Most recently, He and Yuan [arXiv:2108.08554, 2021] have proposed a balanced augmented Lagrangian method (ALM) for the canonical convex programming problem with linear constraints, which advances the original ALM by balancing its subproblems and improving its implementation. In this short note, we propose a dual-primal version of the  balanced ALM, which updates the new iterate via a conversely dual-primal iterative order formally. The proposed method inherits all advantages of the prototype balanced ALM, and its convergence analysis can be well conducted in the context of variational inequalities. In addition, its numerical efficiency is demonstrated by the basis pursuit problem.

\medskip


 \medskip

  \hrule

  }}

\bigskip

\section{Introduction}

A basic optimization model is the classic convex programming problem with linear equality constraints:
\begin{equation}\label{problem}
  \min\big\{\theta(x)   \mid  Ax=b,  \;  x\in\mathcal{X} \big\},
\end{equation}
where $\theta:\Re^{n}\to {\Re}$ is a proper convex but not necessarily smooth function, $\mathcal{X}\subseteq\Re^n$ is a closed convex set,  $A\in\Re^{m\times n}$ and $b\in\Re^m$. Among algorithms for solving   \eqref{problem}, the augmented Lagrangian method (ALM) proposed in \cite{Hes69,Powell69} turns out to be fundamental, and it plays a significant role in both theoretical study and algorithmic design for various convex programming problems. We refer to, e.g., \cite{Bertsekas1982,birgin2014practical,FG1983,glowinski1989augmented,ito2008lagrange} for the vast volume of related literature. In particular, it was shown in \cite{Rock76,Rock76B} that the ALM can be interpreted as an application of the proximal point algorithm (PPA) introduced in \cite{Mar70}.   In practice, with given $(x^k,\lambda^k)$, the original ALM generates the new iterate $(x^{k+1},\lambda^{k+1})$ via
\begin{equation}\label{ALM}
  \left\{
    \begin{array}{rll}
      x^{k+1} &=& \arg\min\big\{ \mathcal{L}_{\beta}(x, \lambda^k)   \mid  x\in\mathcal{X}\big\},\\[0.2cm]
      \lambda^{k+1} &=& \lambda^k-\beta(Ax^{k+1}-b),
    \end{array}
  \right.
\end{equation}
where  $\beta>0$ is the penalty parameter for the linear constraints, $\lambda\in\Re^m$ is the Lagrangian multiplier and
$$ \mathcal{L}_{\beta}(x, \lambda):=\theta(x)-\lambda^T(Ax-b) + \frac{\beta}{2}\|Ax-b\|_2^2$$
is the corresponding augmented  Lagrangian function of \eqref{problem}. Throughout our discussion, the parameter $\beta$ is assumed to be fixed for simplification, and hereafter, we also call $x$ and $\lambda$ the primal and dual variables, respectively.

Ignoring some constant terms, it is trivial to see that the essential step for implementing the original ALM \eqref{ALM} equals to the minimization problem
\begin{equation}\label{xcore}
  x^{k+1}=\arg\min\bigg\{\theta(x)+\frac{\beta}{2}\big\|Ax-(b+\frac{1}{\beta}\lambda^k)\big\|_2^2   \;\;\big|\;\; x\in\mathcal{X}\bigg\}.
\end{equation}
Obviously, the solution set of \eqref{xcore} is essentially determined by the objective function $\theta$, the matrix $A$ and the domain $\mathcal{X}$ in \eqref{problem}.  To improve the implementation of \eqref{ALM}, the so-named linearized ALM has attracted a wide of attention in the literature (see, e.g., \cite{HMY2020,HeYuan2021,YY2013}). Moreover, as discussed in \cite{HeYuan2021}, the linearized ALM for \eqref{problem} can be stated as
 \begin{subequations} \label{LALM}
\begin{numcases}{}
\label{LA-X}  x^{k+1} =\arg\min\bigg\{ \theta(x)+\frac{r}{2}\big\|x-[x^k+\frac{1}{r}A^T(\lambda^k-\beta(Ax^k-b))]\big\|_2^2   \;\;\big|\;\;  x\in\mathcal{X}\bigg\},\\[0.2cm]
 \lambda^{k+1} = \lambda^k-\beta(Ax^{k+1}-b),
\end{numcases}
\end{subequations}
where the parameters $r>0$ and $\beta>0$ need to satisfy the condition $r>\beta\rho(A^TA)$ to theoretically ensure the convergence of \eqref{LALM}. Here,  $\rho(\cdot)$ is the spectrum radius of a matrix.  Note that the matrix $A$ is decoupled in \eqref{LA-X}. The reshaped subproblem \eqref{LA-X} is thus easier to implement than \eqref{xcore}. In particular, it reduces to the proximity operator of $\theta$ when $\mathcal{X}=\Re^n$, which generally has a closed-form solution for some special cases (e.g., $\theta$ is a quadratic or norm function).  We  refer to, e.g.,  \cite{Candes,Chen,Parikh,YY2013} for these particular application scenarios  arising  in data science communities.

There is a structural restriction $r>\beta\rho(A^TA)$ in the well-reshaped linearized ALM \eqref{LALM}.  For a fixed $\beta>0$, it is clear that the quadratic term in \eqref{LA-X} will dominate the objective function of  \eqref{LA-X} if $\rho(A^TA)$ is two large, which would result in a tiny step size and thus limit the numerical efficiency of \eqref{LALM}. To reduce such a restriction, it was proved in \cite{HMY2020} that this restriction can be decreased to $r>0.75\beta\rho(A^TA)$ ($0.75$ is also the optimal bound) by using an indefinite proximal regularization technique, which allows a bigger step size and thus potentially accelerates the convergence. Most recently,  a balanced ALM has been presented in \cite{HeYuan2021}, which has no such a restriction and takes the following iterative scheme:
\begin{subequations} \label{PDALM}
\begin{numcases}{}
\label{PDALM-X} x^{k+1} = \arg\min\bigg\{ \theta(x) + \frac{\beta}{2}\big\|x-(x^k+\frac{1}{\beta}A^T\lambda^k)\big\|_2^2   \;\;\big|\;\;  x\in\mathcal{X}\bigg\}, \\[0.1cm]
\label{PDALM-Y} \lambda^{k+1} = \lambda^k-(\frac{1}{\beta}AA^T+\delta I_m)^{-1}\big[A(2x^{k+1}-x^k)-b\big],
\end{numcases}
\end{subequations}
where $\beta>0$ and $\delta>0$ are free parameters. Moreover, as discussed in \cite{HeYuan2021}, the parameter $\delta>0$ is merely used to ensure the positive definiteness the induced matrix theoretically and it can be just fixed as a small value beforehand. It only needs to empirically and technically tune the parameter $\beta$ when implementing \eqref{PDALM}. Clearly, compared with the original ALM \eqref{ALM}, the balanced ALM \eqref{PDALM} enjoys great advantages in mainly two fields: first, the primal subproblem \eqref{PDALM-X} is easier to implement; second, there is no additional limitation on $\rho(A^TA)$. At the same time, we need to note that the dual subproblem \eqref{PDALM-Y} becomes slightly complicated, because the inverse of $\frac{1}{\beta}AA^T+\delta I_m$ is required beforehand. Fortunately,  it can be found easily, e.g., by the Cholesky decomposition.

The primary purpose of this note is to present a dual-primal version of the balanced ALM \eqref{PDALM} for the linearly constrained convex programming problem \eqref{problem}. More concretely,  our new method takes the following iterative scheme:
\begin{equation}\label{DPALM}
  \left\{
    \begin{array}{lcl}
 \bar{\lambda}^k & = & \lambda^k-(\frac{1}{\beta}AA^T+\delta I_m)^{-1}(Ax^k-b), \\[0.1cm]
 \bar{x}^k & = & \arg\min\Big\{ \theta(x) + \frac{\beta}{2}\big\|x-[x^k+\frac{1}{\beta}A^T(2\bar{\lambda}^k-\lambda^k)]\big\|_2^2   \;\;\big|\;\;  x\in\mathcal{X}\Big\}, \\[0.3cm]
 x^{k+1} & = & x^k + \alpha(\bar{x}^k-x^k), \\[0.2cm]
 \lambda^{k+1} & = & \lambda^k + \alpha(\bar{\lambda}^k-\lambda^k),
    \end{array}
  \right.
\end{equation}
where $\beta>0$ and $\delta>0$ are free parameters, and $\alpha\in(0,2)$ is the extrapolation parameter. As can be seen easily, the proposed method \eqref{DPALM} generates first the dual variable $\lambda$, then the primal variable  $x$, and it maintains the same computational difficulty with the prototype balanced ALM \eqref{PDALM}.  It is thus named the dual-primal balanced ALM in this short note. Also, the new introduced method \eqref{DPALM} can be easily extended to tackle the more general separable convex programming problem with both linear equality and inequality constraints.  We will  present a generalized dual-primal balanced ALM for more general convex programming models in Section \ref{sec4}.

The rest of this note is organized as follows. In Section \ref{sec2}, we summarize some fundamental results for streamlining our analysis. In Section 3, we show the global convergence of the dual-primal balanced ALM \eqref{DPALM}, along with a worst-case $\mathcal{O}(1/N)$ convergence rate. Moreover, we present a generalized scheme for more general convex programming models in Section \ref{sec4}. The numerical experiment is further conducted in Section \ref{sec5}, which is used to illustrate the efficiency of the proposed method. Finally,  some conclusions are made in Section \ref{sec6}.

\section{Preliminaries}\label{sec2}
\setcounter{equation}{0}
\setcounter{remark}{0}
In this section, we summarize some preliminaries for further analysis. Let us recall first a primary lemma whose proof is elementary and can be found in, e.g., \cite{Beck}.
\begin{lemma} \label{CP-TF}
\begin{subequations} \label{CP-TF0}
Let $f:\Re^{l}\to {\Re}$ and $g:\Re^{l}\to {\Re}$ be convex functions, and ${\cal Z}\subseteq \Re^l$ be a closed convex set. If $g$ is differentiable on an open set containing ${\cal Z}$ and the solution set of the minimization problem
$\min\{f(z) + g(z) \mid z\in {\cal Z}\}$ is nonempty, then we have
\begin{equation}\label{CP-TF1}
  z^*  \in \arg\min \big\{ f(z) + g(z)  \mid   z\in {\cal Z}\big\}
\end{equation}
if and only if
\begin{equation}\label{CP-TF2}
  z^*\in {\cal Z}, \quad   f(z) - f(z^*) + (z-z^*)^T\nabla g(z^*) \ge 0, \quad  \forall \;  z \in {\cal Z}.
\end{equation}
\end{subequations}
\end{lemma}

\subsection{Variational inequality reformulation of (\ref{problem})}\label{sec2.1}

Following the analogous technique in, e.g., \cite{GHY2014,HeYuanSIAMIS,HeYuan-SIAM-N,HeYuan2021}, our analysis will be conducted in the variational inequality (VI) context.
Let us first write the VI reformulation for the optimal condition of the studied model \eqref{problem}.

Let $\Omega:=\mathcal{X}\times\Re^m$  and the Lagrangian function of \eqref{problem} be defined as
\begin{equation}\label{lag}
  L(x,\lambda)  = \theta(x)-\lambda^T(Ax-b),
\end{equation}
where $\lambda\in\Re^m$ is the associated  Lagrangian multiplier. The pair $(x^*,\lambda^*)\in \Omega$ is called a saddle point of  \eqref{lag} if it satisfies
\begin{equation}\label{VI-Chara}
  L_{\lambda\in\Re^m}(x^\ast,\lambda)\leq L(x^\ast,\lambda^\ast) \leq L_{x\in\mathcal{X}}(x,\lambda^\ast).
\end{equation}
That is,
$$
\left\{ \begin{array}{l}
     x^*  \in \hbox{argmin} \{  L(x,\lambda^*) \mid  x\in {\cal X} \},   \\[0.2cm]
     \lambda^*  \in \hbox{argmax} \{L(x^*,\lambda)   \mid  \lambda\in \Re^m\}.
\end{array} \right.
$$
Then, according to Lemma \ref{CP-TF},  the above  inequalities can be alternatively rewritten as
$$\left\{ \begin{array}{lrl}
   x^*\in {\cal X}, &   \theta(x) -  \theta(x^*) + (x-x^*)^T(- A^T\lambda^*) \ge 0, & \forall  \; x\in {\cal X}, \\[0.2cm]
   \lambda^*\in \Re^m,   &      (\lambda-\lambda^*)^T(Ax^*-b)\ge 0,  &  \forall \;   \lambda\in \Re^m,
\end{array} \right.$$
or more compactly,
\begin{subequations}\label{VI}
\begin{equation}\label{VI-S}
  \hbox{VI}(F,\theta,\Omega): \quad w^*\in \Omega, \quad \theta(x) -\theta(x^*) + (w-w^*)^T F(w^*) \ge 0, \quad \forall \;  w\in\Omega,
\end{equation}
by setting
\begin{equation}\label{Notation-uFO}
  w = \left(\!\!\begin{array}{c}
                   x\\
                   \lambda \end{array}\!\! \right),\quad
  F(w) =\left(\!\!\begin{array}{c}
     - A^T\lambda \\
     Ax-b \end{array}\!\! \right)
    \quad \hbox{and} \quad  \Omega= {\cal X} \times \Re^m.
\end{equation}
\end{subequations}
Note that the operator $F$ defined in \eqref{Notation-uFO} is affine with a skew-symmetric matrix. It holds that
\begin{equation}\label{EQF}
  (u- v)^T(F(u) -F(v))\equiv0, \quad \forall \; u,\,v \in \Re^{(n+m)},
\end{equation}
which indicates that $F$ is monotone. Throughout our discussion, we denote by $\Omega^*$ the solution set of the VI \eqref{VI}, which is also the solution set of the studied model \eqref{problem}.

\subsection{Prediction-correction interpretation of (\ref{DPALM})}

The artificial prediction-correction interpretation for a known algorithm is a powerful technique for streamlining its convergence analysis, and the related works can be found in, e.g., \cite{GHY2014,HeYuanSIAMIS,HeYuan2021}.
To simplify the convergence analysis of the dual-primal balanced ALM \eqref{DPALM}, we also interpret it into a prediction-correction-type method as follows.
\begin{center}
 \fbox{\begin{minipage}{16.3cm}
\medskip
\textbf{(Prediction step)} With given $(x^k,\lambda^k)$, the dual-primal balanced ALM \eqref{DPALM} begins with

\vspace{-0.5cm}
\begin{subequations} \label{DPALM-P}
\begin{numcases}{}
\label{DPALMy}
 \bar{\lambda}^k = \lambda^k-(\frac{1}{\beta}AA^T+\delta I_m)^{-1}(Ax^k-b), \\[-0.1cm]
\label{DPALMx}
 \bar{x}^k = \arg\min\Big\{ \theta(x) + \frac{\beta}{2}\big\|x-[x^k+\frac{1}{\beta}A^T(2\bar{\lambda}^k-\lambda^k)]\big\|_2^2   \;\;\big|\;\;  x\in\mathcal{X}\Big\}.
\end{numcases}
\end{subequations}
\textbf{(Correction step)} Then, with $(\bar{x}^k,\bar{\lambda}^k)$ as a predictor, it further updates the new iterate $(x^{k+1},\lambda^{k+1})$ via
\begin{equation}\label{DPALM-C}
  \left(
    \begin{array}{c}
      x^{k+1} \\
      \lambda^{k+1} \\
    \end{array}
  \right) =  \left(
    \begin{array}{c}
      x^{k} \\
      \lambda^{k} \\
    \end{array}
  \right) - \alpha
 \left(
    \begin{array}{c}
      x^{k}-\bar{x}^k \\
      \lambda^{k} -\bar{\lambda}^k \\
    \end{array}
  \right),
\end{equation}
where $\alpha\in(0,2)$ is the extrapolation parameter.
\medskip
\end{minipage}
}
\end{center}

\section{Convergence analysis}\label{sec3}
In this section,  we establish the convergence analysis for the dual-primal balanced ALM \eqref{DPALM}, which is rooted in the prediction-correction interpretation \eqref{DPALM-P}-\eqref{DPALM-C}. Let us first prove two pivotal lemmas.
\setcounter{equation}{0}
\setcounter{remark}{0}
\begin{lemma}\label{kle}
Let $\bar{w}^k = (\bar{x}^k, \bar{\lambda}^k)$ be the predictor generated  by the prediction step \eqref{DPALM-P} with given $w^k =(x^k, \lambda^k)$. Then, we get
\begin{equation}\label{Pred1}
  \theta(x) -\theta(\bar{x}^{k}) + (w-\bar{w}^{k})^T F(\bar{w}^{k}) \ge (w-\bar{w}^{k})^T  H(w^k  -\bar{w}^{k}), \quad  \forall \; w\in  \Omega,
\end{equation}
where
\begin{equation}\label{H}
  H=\left(
  \begin{array}{cc}
    \beta I_n & -A^T \\
    -A & \frac{1}{\beta}AA^T+\delta I_m \\
  \end{array}
\right).
\end{equation}
\end{lemma}
\begin{proof}
To begin with, for the subproblem \eqref{DPALMy}, we have
$$Ax^k-b + (\frac{1}{\beta}AA^T+\delta I_m) (\bar{\lambda}^k - \lambda^k) = 0,$$
which is also equivalent to
\begin{equation}\label{DALMD}
  \bar{\lambda}^k  \in \Re^m,   \quad (\lambda-\bar{\lambda}^k )^T \big\{  A\bar{x}^k-b - A(\bar{x}^k-x^k) + (\frac{1}{\beta}AA^T+\delta I_m) (\bar{\lambda}^k - \lambda^k)\big\} \ge 0, \quad  \forall \; \lambda\in \Re^m.
\end{equation}
For the subproblem \eqref{DPALMx}, it follows from  Lemma \ref{CP-TF} that
 $$ \bar{x}^{k}\in {\cal X},   \;\;  \theta(x) -\theta(\bar{x}^{k}) + (x-\bar{x}^{k})^T \{-A^T(2\bar{\lambda}^k-\lambda^k) + \beta(\bar{x}^{k}-x^k) \} \ge 0, \quad  \forall \; x\in {\cal X}, $$
which can be further rewritten as
\begin{equation}\label{DALMP}
\bar{x}^{k}\in {\cal X}, \quad \theta(x) -\theta(\bar{x}^{k}) + (x-\bar{x}^{k})^T \{-A^T\bar{\lambda}^k + \beta(\bar{x}^{k}-x^k)-A^T(\bar{\lambda}^k-\lambda^k) \} \ge 0, \quad  \forall \; x\in {\cal X}.
\end{equation}
Adding \eqref{DALMD} and \eqref{DALMP} together, we have
\begin{eqnarray*} 
   \lefteqn{(\bar{x}^k,\bar{\lambda}^k)\in{\cal X} \times \Re^m, \quad \theta(x) - \theta(\bar{x}^{k})
     + \left(\begin{array}{c}
    x- \bar{x}^{k} \\[0.1cm]
     \lambda - \bar{\lambda}^k
     \end{array}\right)^T
     \left\{
     \left(\begin{array}{c}
      - A^T\bar{\lambda}^k  \\[0.1cm]
       A\bar{x}^{k}-b
     \end{array}\right) \right. }\nn \\
     & &\qquad + \left.
  \left(\begin{array}{c}
      \beta(\bar{x}^{k} -x^k)  - A^T(\bar{\lambda}^k  -\lambda^k)\\[0.1cm]
       -A(\bar{x}^k-x^k)   +  (\frac{1}{\beta}AA^T+\delta I_m)(\bar{\lambda}^k  - \lambda^k)
     \end{array}\right)  \right\} \ge 0, \quad \forall \; (x,\lambda)\in
      {\cal X} \times \Re^m.
  \end{eqnarray*}
Using the notation in \eqref{Notation-uFO} and the matrix $H$ defined in \eqref{H}, the assertion of this lemma follows immediately.
\end{proof}

At the same time, the positive definiteness of the induced matrix $H$ defined in \eqref{H} can be ensured by the following proposition.
\begin{proposition}
The matrix $H$ defined in \eqref{H}  is positive definite for any $\beta>0$ and $\delta>0$.
\end{proposition}
\begin{proof}
First of all, it is trivial to verify that
$$H=\left(
  \begin{array}{cc}
    \beta I_n & -A^T \\
    -A & \frac{1}{\beta}AA^T+\delta I_m \\
  \end{array}
\right)=\left(
          \begin{array}{c}
            -\sqrt{\beta} I_n \\[0.1cm]
            \sqrt{\frac{1}{\beta}}A \\
          \end{array}
        \right)\left(
                 \begin{array}{cc}
                   -\sqrt{\beta} I_n & \sqrt{\frac{1}{\beta}}A^T \\
                 \end{array}
               \right)
+\left(
  \begin{array}{cc}
    0 & 0 \\
    0 & \delta I_m \\
  \end{array}
\right).$$
Then, for any $w=(x,\lambda)\neq0$, we have
$$w^THw=\|\sqrt{\frac{1}{\beta}}A^T\lambda-\sqrt{\beta}x\|^2+\delta\|\lambda\|^2>0,$$
and the proof is complete accordingly.
\end{proof}

The next lemma further refines the right-hand side of \eqref{Pred1}, and it is used to quantify the difference of a solution point of
the VI \eqref{VI} by recursively quadratic terms.

\begin{lemma}  Let $\{w^k\}$ and $\{\bar{w}^k\}$ be the sequences generated  by the prediction-correction scheme \eqref{DPALM-P}-\eqref{DPALM-C} with $\beta>0$ and $\delta>0$. Then, for any $\alpha\in(0,2)$, we have
\begin{eqnarray}\label{ki2}
  \lefteqn{\alpha\big\{\theta(x)-\theta(\bar{x}^k)+(w-\bar{w}^k)^TF(w)\big\}}    \nonumber \\
   & \geq &  \frac{1}{2}\big\{\|w-w^{k+1}\|_H^2 - \|w-w^{k}\|_H^2+\alpha(2-\alpha) \|w^k-\bar{w}^{k}\|_H^2\big\}, \quad  \forall \; w \in  \Omega,
\end{eqnarray}
where $H$ is the matrix given by \eqref{H}.
\end{lemma}
\begin{proof}
First of all, it follows from \eqref{Pred1} and $w^{k+1}=w^k-\alpha(w^k-\bar{w}^k)$ (see \eqref{DPALM-C}) that
\begin{equation}\label{pre2}
   \alpha\big\{\theta(x) -\theta(\bar{x}^{k}) + (w-\bar{w}^{k})^T F(\bar{w}^{k})\big\} \ge (w-\bar{w}^{k})^T  H(w^k  - w^{k+1}), \quad  \forall \; w\in  \Omega.
\end{equation}
Applying the identity
$$(a-b)^TH(c-d)=\frac{1}{2}\big\{\|a-d\|_H^2-\|a-c\|_H^2\big\}+\frac{1}{2}\big\{\|c-b\|_H^2-\|d-b\|_H^2\big\}$$
to the right-hand side of \eqref{pre2} with $a=w,\;b=\bar{w}^k,\;c=w^k\; \hbox{and} \;d=w^{k+1}$, it further implies that
\begin{eqnarray}\label{pre3}
\lefteqn{(w-\bar{w}^k)^TH(w^k-{w}^{k+1})} \nonumber\\
 &\quad =& \frac{1}{2}\big\{\|w-w^{k+1}\|_H^2-\|w-w^k\|_H^2\big\} +\frac{1}{2}\big\{\|w^k-\bar{w}^k\|_H^2-\|w^{k+1}-\bar{w}^k\|_H^2\big\}.
\end{eqnarray}
For the second term of right-hand side of \eqref{pre3}, it follows from \eqref{DPALM-C} that
\begin{eqnarray}\label{pre4}
  \frac{1}{2}\big\{\|w^k-\bar{w}^k\|_H^2-\|w^{k+1}-\bar{w}^k\|_H^2\big\} &=& \frac{1}{2}\big\{\|w^k-\bar{w}^k\|_H^2-\|w^k-\alpha(w^k-\bar{w}^k)-\bar{w}^k\|_H^2\big\}  \nonumber \\
  &=& \frac{1}{2}\alpha(2-\alpha)\|w^k-\bar{w}^k\|_H^2.
\end{eqnarray}
Then, combining with \eqref{pre3} and \eqref{pre4}, the inequality \eqref{pre2} equals to
$$\begin{aligned}
&\alpha\big\{\theta(x)-\theta(\bar{x}^k)+(w-\bar{w}^k)^TF(\bar{w}^k)\big\}\\
&\qquad \geq \frac{1}{2}\big\{\|w-w^{k+1}\|_H^2 - \|w-w^{k}\|_H^2+\alpha(2-\alpha) \|w^k-\bar{w}^{k}\|_H^2\big\}, \quad  \forall \; w \in  \Omega.
\end{aligned}$$
Note that $(w-\bar{w}^k)^TF(\bar{w}^k)\equiv(w-\bar{w}^k)^TF(w)$ (see \eqref{EQF}). The assertion of this lemma follows immediately.
\end{proof}

With the help of the above lemmas, the strict contraction of the sequence $\{w^k\}$ generated by the proposed method \eqref{DPALM} can be summarized in the following theorem.
\begin{theorem}
Let $\{w^k\}$ and $\{\bar{w}^k\}$ be the sequences generated  by the prediction-correction scheme \eqref{DPALM-P}-\eqref{DPALM-C} with $\beta>0$ and $\delta>0$. Then, for arbitrary $\alpha\in(0,2)$, it holds that
\begin{equation}\label{ki3}
  \|w^{k+1}-w^\ast\|_H^2\leq \|w^{k}-w^\ast\|_H^2-\alpha(2-\alpha) \|w^k-\bar{w}^k\|_H^2, \quad  \forall \; w^\ast\in  \Omega^\ast,
\end{equation}
where $H$ is the matrix defined in \eqref{H}.
\end{theorem}
\begin{proof}
Setting $w$ in \eqref{ki2} as arbitrary $w^\ast\in\Omega^\ast$, we have
\begin{eqnarray}\label{pre5}
  \lefteqn{\|w^{k}-w^\ast\|_H^2-\|w^{k+1}-w^\ast\|_H^2-\alpha(2-\alpha) \|w^k-\bar{w}^{k}\|_H^2}  \nonumber \\[0.1cm]
   & \geq & 2\alpha\big\{\theta(\bar{x}^k)-\theta(x^\ast)+(\bar{w}^k-w^\ast)^TF(w^\ast)\big\}, \quad  \forall \; w^\ast \in  \Omega^\ast.
\end{eqnarray}
Since $w^\ast\in\Omega^\ast$ and $\bar{w}^k\in\Omega$, it follows from \eqref{VI-S} that the right-hand side of \eqref{pre5} is non-negative. This leads to the assertion of the theorem immediately.
\end{proof}
Based on the essential contraction property \eqref{ki3},  the global convergence of the dual-primal balanced ALM \eqref{DPALM} can be shown in the following theorem.
\begin{theorem}\label{essetheorem}
The sequence $\{w^k\}$ generated by the dual-primal balanced ALM \eqref{DPALM} converges to some $w^\infty\in\Omega^\ast$ for any $\beta>0$, $\delta>0$ and $\alpha\in(0,2)$.
\end{theorem}
\begin{proof} To begin with, it follows from the inequality \eqref{ki3} that the sequence $\{w^k\}$ is bounded. Summarizing \eqref{ki3} over $k=0,1,\ldots,\infty$, it further implies that
$$\sum_{k=0}^{\infty}\alpha(2-\alpha)\|w^k -\bar{w}^k\|_H^2\leq\|w^0 -w^\ast\|_H^2.$$
Therefore, we have
\begin{equation}\label{Contrac-uut}
  \lim_{k\to \infty}\|w^k -\bar{w}^k\|_H^2=0,
\end{equation}
which means that the sequence $\{\bar{w}^k\}$ is also bounded. Let $w^{\infty}$ be a cluster point of $\{\bar{w}^k\}$ and $\{\bar{w}^{k_j}\}$ be a subsequence converging to $w^{\infty}$. Then, according to \eqref{Pred1}, we have
$$\bar{w}^{k_j}\in \Omega, \quad \theta(x)-\theta(\bar{x}^{k_j}) + (w-\bar{w}^{k_j})^TF(\bar{w}^{k_j}) \ge (w-\bar{w}^{k_j})^TH(w^{k_j}-\bar{w}^{k_j}), \quad \forall\; w\in \Omega.$$
Note that the matrix $H$ defined in \eqref{H} is non-singular. It follows from \eqref{Contrac-uut} and the continuity of $\theta$ and $F$ that
$$ w^{\infty}\in \Omega, \quad \theta(x)-\theta(x^{\infty}) + (w- w^{\infty})^T F(w^{\infty}) \ge 0, \quad \forall\; w\in \Omega.$$
This indicates that $w^{\infty}\in\Omega^\ast$, which is also a solution point of the studied model \eqref{problem}. Moreover, it follows from \eqref{Contrac-uut} that $\lim_{k\rightarrow\infty}w^{k_j}=w^\infty$. In addition, according to \eqref{ki3}, we have
$$\|w^{k+1} - w^{\infty}\|_H^2 \le \|w^k  - w^{\infty}\|_H^2,$$
which means that it is impossible that the sequence $\{w^k\}$ has more than one cluster point. Consequently, we have $\lim_{k\rightarrow\infty}w^k=w^\infty\in\Omega^\ast$ and the proof is complete.
\end{proof}

\begin{remark}
Following the similar analysis technique in, e.g., \cite{HeYuanSIAMIS,HeYuan-SIAM-N,HY2015,HeYuan2021},  it is trivial to show that the dual-primal balanced ALM \eqref{DPALM} also enjoys a worst-case $\mathcal{O}(1/N)$ convergence rate in both ergodic and point-wise sense, where $N$ is the iteration counter.
\end{remark}

\section{Extensions to more general models}\label{sec4}
\setcounter{remark}{0}
\setcounter{equation}{0}

In this section, we extend the dual-primal balanced ALM \eqref{DPALM} to solve the following more general separable convex programming problem with  linear equality or inequality constraints:
\begin{equation}\label{Mp}
   \begin{array}{ll}
    \min & \sum_{i=1}^p \theta_i(x_i)    \\[0.2cm]
    \;\;\hbox{s.t.}      & \sum_{i=1}^p A_i x_i =b\; (\hbox{or}\geq b),  \\[0.2cm]
                                &  x_i\in {\cal X}_i, \;\; i=1,\ldots, p,
          \end{array}
\end{equation}
where $\theta_i: {\Re}^{n_i}\to {\Re} \;(i=1,\ldots, p)$ are closed proper convex but not necessarily smooth functions, ${\cal X}_i\subseteq \Re^{n_i} \;(i=1,\ldots, p)$ are closed convex sets, $A_i\in\Re^{m\times n_i}\;(i=1,\ldots,p)$  and $b\in \Re^m$. The possible applications of the model \eqref{Mp} can be found in, e.g., \cite{boyd2010distributed,chandrasekaran2012latent,McLachlan,Sun2021,yuan2010continuous}.

To unify the notation, let us first define
\begin{equation}\label{domains}
\Lambda=\left\{
\begin{array}{ll}
\Re^m, & \hbox{if } \sum_{i=1}^p A_i x_i =b,\\[0.2cm]
\Re_+^m, & \hbox{if } \sum_{i=1}^p A_i x_i \geq b.
\end{array}
\right.
\end{equation}
It is clear that the basic model \eqref{problem} coincides with the case of \eqref{Mp} where $p=1$ and $\Lambda=\Re^m$.
Then, a generalized dual-primal balanced ALM for the more general convex programming problem \eqref{Mp} is proposed as follows.
\begin{center}\fbox{
 \begin{minipage}{16.1cm}
 \medskip
 \noindent{\bf Algorithm: a generalized dual-primal balanced ALM for (\ref{Mp}) }
\medskip

Let $\beta_i>0$ $(i=1,\ldots,p)$ and $\delta>0$  be any constants, and we define
\begin{equation}\label{M}
    M_p = \sum_{i=1}^p\frac{1}{\beta_i}A_iA_i^T+\delta I_m.
\end{equation}
Then, the generalized dual-primal balanced ALM for \eqref{Mp} includes the following two steps:

\textbf{(Prediction step)} With given $(x_1^k,\ldots,x_p^k,\lambda^k)$,  it first generates $(\bar{x}_1^k,\ldots,\bar{x}_p^k,\bar{\lambda}^k)$ via
\begin{subequations} \label{DPALMm}
\begin{numcases}{}
\label{DPALMmy}
 \bar{\lambda}^k =\arg\min_{\lambda\in \Lambda}\Big\{\frac{1}{2}(\lambda-\lambda^k)^TM_p(\lambda-\lambda^k)+\lambda^T(\sum_{i=1}^pA_ix_i^k-b)\Big\}, \\
\label{DPALMmx}
 \bar{x}_i^k = \arg\min_{x_i\in\mathcal{X}_i}\Big\{ \theta_i(x_i) + \frac{\beta_i}{2}\big\|x_i-[x_i^k+\frac{1}{\beta_i}A_i^T(2\bar{\lambda}^k-\lambda^k)]\big\|_2^2\Big\}, \; i=1,\ldots,p.
\end{numcases}
\textbf{(Correction step)} Then, with $(\bar{x}_1^k,\ldots,\bar{x}_p^k,\bar{\lambda}^k)$  as a predictor,  it further updates the new iterate $(x_1^{k+1},\ldots,x_p^{k+1},\lambda^{k+1})$ by
\begin{equation}\label{DPALMm-C}
   \left(
    \begin{array}{c}
      x_1^{k+1} \\[-0.1cm]
      \vdots \\[-0.1cm]
      x_p^{k+1} \\
      \lambda^{k+1} \\
    \end{array}
  \right) =  \left(
    \begin{array}{c}
      x_1^{k} \\[-0.1cm]
      \vdots \\[-0.1cm]
      x_p^{k} \\
      \lambda^{k} \\
    \end{array}
  \right) - \alpha
 \left(
    \begin{array}{c}
      x_1^k-\bar{x}_1^k \\[-0.1cm]
      \vdots \\[-0.1cm]
      x_p^k-\bar{x}_p^k  \\
      \lambda^{k} -\bar{\lambda}^k \\
    \end{array}
  \right),
\end{equation}
\end{subequations}
where $\alpha\in(0,2)$ is the extrapolation parameter.
\medskip
 \end{minipage}}
\end{center}

\begin{remark}
Note that the $\lambda$-subproblem in \eqref{DPALM} is equivalent to the minimization problem
$$\bar{\lambda}^k = \arg\min\Big\{ \frac{1}{2}(\lambda-\lambda^k)^T\big[\frac{1}{\beta}AA^T+\delta I_m\big](\lambda-\lambda^k)+\lambda^T(Ax^k-b) \;\;\big|\;\; \lambda\in\Re^m    \Big\}.$$
The elementary dual-primal balanced ALM \eqref{DPALM} is a special case of \eqref{DPALMm} with $p=1$ and $\Lambda=\Re^m$.
\end{remark}

\begin{remark}
When the inequality-constrained case of \eqref{Mp} is considered, the subproblem \eqref{DPALMmy} would reduce to a standard quadratic programming with non-negative sign constraints:
$$\min\Big\{\frac{1}{2}(\lambda-\lambda^k)^TM_p(\lambda-\lambda^k)+\lambda^T(\sum_{i=1}^pA_ix_i^k-b)  \;\mid\;  \lambda\in \Re_+^m \Big\}.$$
As discussed in \cite{HeYuan2021}, such a minimization problem can be efficiently solved by many well-known solvers such as conjugate gradient method and Lemke algorithm  (see, e.g., \cite{golub,Nocedal}).
\end{remark}

\subsection{VI reformulation of (\ref{Mp})}
 Similarly as Section \ref{sec2}, to simplify the analysis for the more general model \eqref{Mp},  we first derive its optimal condition  in the VI context.

Let $\Omega:=\mathcal{X}_1\times\cdots\times\mathcal{X}_p\times\Lambda$ and the Lagrangian function of \eqref{Mp} be defined as
\begin{equation}\label{lagm}
  L(x_1,\ldots,x_p,\lambda)  = \sum_{i=1}^p\theta_i(x_i)-\lambda^T(\sum_{i=1}^pA_ix_i-b).
\end{equation}
Again,  it is trivial to see that the optimal condition of \eqref{Mp} is equivalent to finding a saddle point $w^\ast=(x_1^\ast,x_2^\ast,\ldots,x_p^\ast,\lambda^\ast)\in\Omega$ of \eqref{lagm} such that
\begin{equation}\label{VIm}
    \left\{ \begin{array}{ll}
     \theta_1(x_1) - \theta_1(x_1^\ast) + (x_1-x_1^\ast)^T(- A_1^T\lambda^\ast) \ge 0, & \forall\; x_1\in {\cal X}_1, \\[0.2cm]
     \theta_2(x_2) - \theta_2(x_2^\ast) + (x_2-x_2^\ast)^T(- A_2^T\lambda^\ast) \ge 0, & \forall\; x_2\in {\cal X}_2, \\
       \qquad \quad \qquad \qquad  \vdots &\\
      \theta_p(x_p) - \theta_p(x_p^\ast) + (x_p-x_p^\ast)^T(- A_p^T\lambda^\ast) \ge 0,  & \forall\; x_p\in {\cal X}_p, \\[0.2cm]
     (\lambda-\lambda^\ast)^T(\sum_{i=1}^pA_ix_i^\ast-b)\geq 0,  &  \forall \; \lambda\in \Lambda,
        \end{array} \right.
\end{equation}
which is also equivalent to the following VI:
\begin{subequations}\label{VIP}
\begin{equation}\label{OVIm}
  \hbox{VI}(\Omega,F,\theta): \quad  w^*\in \Omega, \quad \theta(x) -\theta(x^*) + (w-w^*)^T F(w^*) \ge 0, \quad  \forall  \; w\in\Omega,
\end{equation}
where
\begin{equation}\label{VI-S1}
 w = \left(
                                   \begin{array}{c}
                                    x_1 \\[-0.1cm]
                                     \vdots \\[-0.1cm]
                                     x_p \\
                                     \lambda \\
                                   \end{array}
                                 \right), \quad
F(w)=\left(
                                   \begin{array}{c}
                                     -A_1^T\lambda \\[-0.1cm]
                                     \vdots \\[-0.1cm]
                                     -A_p^T\lambda \\
                                     \sum_{i=1}^pA_ix_i - b \\
                                   \end{array}
                                 \right)  \quad \hbox{and} \quad \Omega=\mathcal{X}_1\times\cdots\times\mathcal{X}_p\times\Lambda.
\end{equation}
\end{subequations}
Again, we denote by $\Omega^\ast$ the solution set of the VI \eqref{VIP}, which is also the solution set of the more general model \eqref{Mp}.

\subsection{Convergence analysis for (\ref{DPALMm})}
Following the same analysis routine in Section \ref{sec3}, we establish the convergence analysis for the generalized dual-primal balanced ALM \eqref{DPALMm} in this subsection. Reusing the same letters in Section \ref{sec3}, we only need to extend Lemma \ref{kle} to a more general case.

\begin{lemma}\label{lema}
Let $M_p$ be the matrix defined in \eqref{M}, and $\bar{w}^k = (\bar{x}_1^k, \ldots,\bar{x}_p^k, \bar{\lambda}^k)$ be the predictor generated  by the prediction step \eqref{DPALMmy}-\eqref{DPALMmx} with given $w^k =(x_1^k,\ldots,x_p^k, \lambda^k)$. Then, we get
\begin{equation}\label{Pred1m}
  \theta(x) -\theta(\bar{x}^{k}) + (w-\bar{w}^{k})^T F(\bar{w}^{k}) \ge (w-\bar{w}^{k})^T  H(w^k  -\bar{w}^{k}), \quad  \forall \; w\in  \Omega,
\end{equation}
where
\begin{equation}\label{HN}
  H=\left(
  \begin{array}{ccccc}
    \beta_1 I_{n_1} & 0  &   \cdots & 0  &  -A_1^T \\
    0 & \ddots &  \ddots & \vdots & \vdots  \\
    \vdots & \ddots &  \ddots & 0 & \vdots  \\
   0 & \cdots & 0 &  \beta_p I_{n_p} &   -A_p^T \\[0.1cm]
    -A_1 & \cdots  & \cdots & -A_p & M_p \\
  \end{array}
\right).
\end{equation}
\end{lemma}
\begin{proof}
First of all, for the subproblem \eqref{DPALMmy}, it follows from Lemma \ref{CP-TF} that
$$
  \bar{\lambda}^k  \in \Lambda,   \quad (\lambda-\bar{\lambda}^k )^T\Big\{  \sum_{i=1}^pA_ix_i^k-b  +M_p(\bar{\lambda}^k - \lambda^k)\Big\} \ge 0, \quad  \forall \; \lambda\in \Lambda,
$$
which can  be further rewritten as
\begin{equation}\label{DPALMv1}
  \bar{\lambda}^k  \in \Lambda,   \quad (\lambda-\bar{\lambda}^k )^T\Big\{  \sum_{i=1}^pA_i\bar{x}_i^k-b
    - \sum_{i=1}^pA_i(\bar{x}_i^k-x_i^k) +M_p(\bar{\lambda}^k - \lambda^k)\Big\} \ge 0, \quad  \forall \; \lambda\in \Lambda.
\end{equation}
Similarly, for each $x_i$-subproblem in \eqref{DPALMmx}, it follows from Lemma \ref{CP-TF} that
 $$ \bar{x}_i^{k}\in {\cal X}_i,   \;\;  \theta_i(x_i) -\theta_i(\bar{x}_i^{k}) + (x_i-\bar{x}_i^{k})^T \{-A_i^T(2\bar{\lambda}^k-\lambda^k) + \beta(\bar{x}_i^{k}-x_i^k) \} \ge 0, \quad  \forall \; x_i\in {\cal X}_i, $$
which also equals to
\begin{eqnarray}\label{DPALMv2}
  \bar{x}_i^k\in {\cal X}_i,   &&  \theta_i(x_i) -\theta_i(\bar{x}_i^{k}) + (x_i-\bar{x}_i^{k})^T  \nonumber \\
    & & \{-A_i^T\bar{\lambda}^k + \beta(\bar{x}_i^{k}-x_i^k)-A_i^T(\bar{\lambda}^k-\lambda^k) \} \ge 0, \quad  \forall \; x_i\in {\cal X}_i.
\end{eqnarray}
Adding \eqref{DPALMv1} and \eqref{DPALMv2} together, and using the notation given in \eqref{VI-S1} and the matrix $H$ defined in \eqref{HN}, the assertion of this lemma follows immediately.
\end{proof}

Again, the positive definiteness of the induced matrix $H$ given by  \eqref{HN} can be guaranteed by the following proposition.

\begin{proposition}
The matrix $H$ defined in \eqref{HN} is positive definite for any $\delta>0$ and $\beta_i>0$ $(i=1,\ldots,p)$.
\end{proposition}
\begin{proof}
First of all, it is trivial to check that
$$\begin{aligned}H&=\left(
  \begin{array}{ccccc}
    \beta_1 I_{n_1} & 0  &   \cdots & 0  &  -A_1^T \\
    0 & \ddots &  \ddots & \vdots & \vdots  \\
    \vdots & \ddots &  \ddots & 0 & \vdots  \\
   0 & \cdots & 0 &  \beta_p I_{n_p} &   -A_p^T \\[0.1cm]
    -A_1 & \cdots  & \cdots & -A_p & \sum_{i=1}^p\frac{1}{\beta_i}A_iA_i^T+\delta I_m \\
  \end{array}
\right)\\[0.2cm]
&=\sum_{i=1}^p\left(
          \begin{array}{c}
         \vdots \\
            -\sqrt{\beta_i} I_{n_i} \\[-0.1cm]
     \vdots \\
            \sqrt{\frac{1}{\beta_i}}A_i \\
          \end{array}
        \right)\left(
                 \begin{array}{cccc}
                  \cdots & -\sqrt{\beta_i} I_{n_i} & \cdots & \sqrt{\frac{1}{\beta_i}}A_i^T  \\
                 \end{array}
               \right)
+\left(
  \begin{array}{cc}
    0 & 0 \\
    0 & \delta I_m \\
  \end{array}
\right).\end{aligned}$$
Then, for arbitrary $w=(x_1,\ldots,x_p,\lambda)\neq0$, we have
$$w^THw=\sum_{i=1}^p\|\sqrt{\frac{1}{\beta_i}}A_i^T\lambda-\sqrt{\beta_i}x_i\|^2+\delta\|\lambda\|^2>0,$$
which further implies that the matrix $H$ is positive definite.
\end{proof}
Beginning with Lemma \ref{lema} and using the same letters, the remaining lemmas and theorems in Section \ref{sec3} then follow accordingly. The convergence analysis for the generalized  dual-primal balanced ALM \eqref{DPALMm} is thus established.

\section{Numerical experiments}\label{sec5}

\setcounter{equation}{0}

In this section, we report the numerical results of the dual-primal balanced ALM \eqref{DPALM} for the classic equality-constrained $l_1$ minimization problem.  The preliminary experimental results show that the proposed method has a significant acceleration compared with some well-known algorithms such as the linearized ALM \eqref{LALM} and the primal-dual algorithm proposed in \cite{CP2011}, and it has an almost same efficiency with the original balanced ALM  \eqref{PDALM}. Our algorithms were written in a Python 3.9 and implemented in a Lenovo computer with  2.20 GHz Intel Core i7-8750H CPU and 16 GB memory.

\subsection{Tested model}
Let us consider the classic equality-constrained $l_1$ minimization problem:
\begin{equation}\label{BP}
  \min\big\{\|x\|_1 \mid Ax=b, \; x\in\Re^n\big\},
\end{equation}
where $\|x\|_1=\sum_{i=1}^n|x_i|$, $A\in\Re^{m\times n}$ ($m<n$) and $b\in\Re^m$. The model \eqref{BP} is also known as the basis pursuit problem, and it plays a significant role in various areas such as compressed sensing and statistical learning. We see, e.g., \cite{BD2009,Chen}  for some survey papers.

 Applying the proposed method \eqref{DPALM} to \eqref{BP}, we have
\begin{equation}\label{DP-BP}
  \left\{
    \begin{array}{cll}
      \bar{\lambda}^k &=& \lambda^k-(\frac{1}{\beta}AA^T+\delta I_m)^{-1}(Ax^k-b), \\[0.2cm]
     \bar{x}^k &=& \arg\min\Big\{ \|x\|_1 + \frac{\beta}{2}\big\|x-[x^k+\frac{1}{\beta}A^T(2\bar{\lambda}^k-\lambda^k)]\big\|_2^2  \;\mid\;  x\in\Re^n\Big\}, \\[0.2cm]
     x^{k+1} & = & x^k + \alpha(\bar{x}^k-x^k), \\[0.1cm]
   \lambda^{k+1} & = & \lambda^k + \alpha(\bar{\lambda}^k-\lambda^k).
    \end{array}
  \right.
\end{equation}
For simplification,  we fix $\alpha=1$ in \eqref{DP-BP}.  Clearly, the $x$-subproblem in \eqref{DP-BP} has a closed-form solution, which can be represented explicitly by the shrinkage operator defined in, e.g., \cite{Chen}. At the same time, as a contrast, we also report the numerical results of the primal-dual algorithm (PDA for short) introduced  in \cite{CP2011},  the linearized ALM \eqref{LALM} and the balanced ALM \eqref{PDALM}. Their associated iterative schemes are trivial and thus skipped for succinctness.

\subsection{Experimental results}
To simulate, we follow the standard way (see, e.g., \cite{Deng2017}) to generate a $x^\ast\in\Re^n$ randomly whose $s$ entries are drawn from the normal distribution $\mathcal{N}(0,1)$ and the rest are zeros. Then, we generate a standard Gaussian matrix $A\in\R^{m\times n}$ whose entries satisfying the normal distribution, and further set $b=Ax^\ast$. In our experiments, we take $m=n/2$, $s=n/10$ and use $(x^0,\lambda^0)=(\textbf{0},\textbf{0})$ as the initial iterate. Moreover, the stopping criterion for \eqref{BP} (see \cite{Deng2017})  is defined as
$$\hbox{ReE}(k):=\frac{\|x^k-x^\ast\|}{\|x^\ast\|}<10^{-7},$$
where ``ReE" is short for the relative error. To implement the aforementioned algorithms efficiently,  we take the specific parameter settings as following:
\begin{itemize}
  \item Algorithm 1: the dual-primal balanced ALM \eqref{DPALM} with $\beta=10$, $\delta=0.001$ and $\alpha=1$;
  \item Algorithm 2: the balanced ALM \eqref{PDALM} with $\beta=10$ and $\delta=0.001$;
  \item Algorithm 3: the PDA with $r=\sqrt{\rho(A^TA)+0.001}$ and $s=\sqrt{\rho(A^TA)+0.001}$;
  \item Algorithm 4: the linearized ALM \eqref{LALM} with $\beta=0.01$ and $r=\beta\rho(A^TA)+0.001$.
\end{itemize}
They are almost optimal for the tested algorithms, selected out of a number of various values.

\begin{table}[H]
\caption{Numerical results for \eqref{BP} solved by the above algorithms. The associated convergence curves on some examples are plotted in Figure \ref{fig1}.}
 \centering
 \begin{tabular}{lccccccccccccc}
  \toprule
  \toprule
 \multirow{2}{*}{$n$} & \multirow{2}{*}{$\rho(A^TA)$} & \multicolumn{2}{c}{Algorithm 1} & \multicolumn{2}{c}{Algorithm 2} & \multicolumn{2}{c}{Algorithm 3} & \multicolumn{2}{c}{Algorithm 4} \cr \cmidrule(lr){3-4} \cmidrule(lr){5-6} \cmidrule(lr){7-8} \cmidrule(lr){9-10}
                 &     &     Iter     &  CPU       &         Iter     &  CPU    &         Iter     &  CPU    &         Iter     &  CPU       \cr
  \midrule
$100$       &    273.01        &    93      &      0.01       &    94      &    0.02    &      298     &    0.05    &      359     &    0.06      \\
$200$       &    571.37        &    98      &      0.02       &    99      &    0.02    &      302     &    0.08    &      325     &    0.08      \\
$300$       &    849.84        &    88      &      0.02       &    89      &    0.02    &      327     &    0.07    &      369     &    0.08      \\
$400$       &    1102.43      &    102    &      0.03       &    103    &    0.03    &      366     &    0.08    &      369     &    0.09      \\
$500$       &    1405.18      &    107    &      0.03       &    107    &    0.03    &      402     &    0.09    &      389     &    0.09      \\
$800$       &    2297.04      &    109    &      0.03       &    110    &    0.03    &      391     &    0.09    &      404     &    0.10      \\
$1000$     &    2875.02      &    160    &      0.05       &    161    &    0.06    &      373     &    0.10    &      371     &    0.11      \\
$2000$     &    5784.24      &    181    &      0.37       &    183    &    0.38    &      429     &    0.64    &      421     &    0.90      \\
$3000$     &    8592.65      &    123    &      0.65       &    123    &    0.65    &      455     &    1.85    &      448     &    2.69      \\
$4000$     &    11735.39    &    191    &      1.91       &    192    &    1.90    &      493     &    3.88    &      493     &    5.66      \\
$5000$     &    14463.51    &    116    &      1.80       &    120    &    1.85    &      499     &    6.09    &      518     &    9.28      \\
$8000$     &    23319.20    &    129    &      5.04       &    137    &    5.54    &      516     &    16.37  &      598     &    28.02    \\
$10000$   &    29150.32    &    183    &      10.98     &    184    &    11.05  &      533     &    25.53  &      652     &    46.74    \\
\bottomrule
 \end{tabular}
 \label{Ta1}
\end{table}

In Table \ref{Ta1}, for various values of $n$,  the spectrum of the matrix $A^TA$ (``$\rho(A^TA)$"),  the required iteration number (``Iter") and the totally computing time in seconds (``CPU") are reported.  It can be seen easily from Table \ref{Ta1} that the proposed method performs competitively with the prototype balanced ALM, and it has a significant acceleration compared with the PDA and the linearized ALM. To further visualize the numerical results,  in Figure \ref{fig1}, we plot the convergence curves versus both iteration numbers and CPU time for the cases where $n=300$ and $n=3000$, which can be  further demonstrated the numerical efficiency of the proposed method.

\begin{figure}[H]
\centering
\subfigure[$n=300$]{
\includegraphics[width=8.0cm]{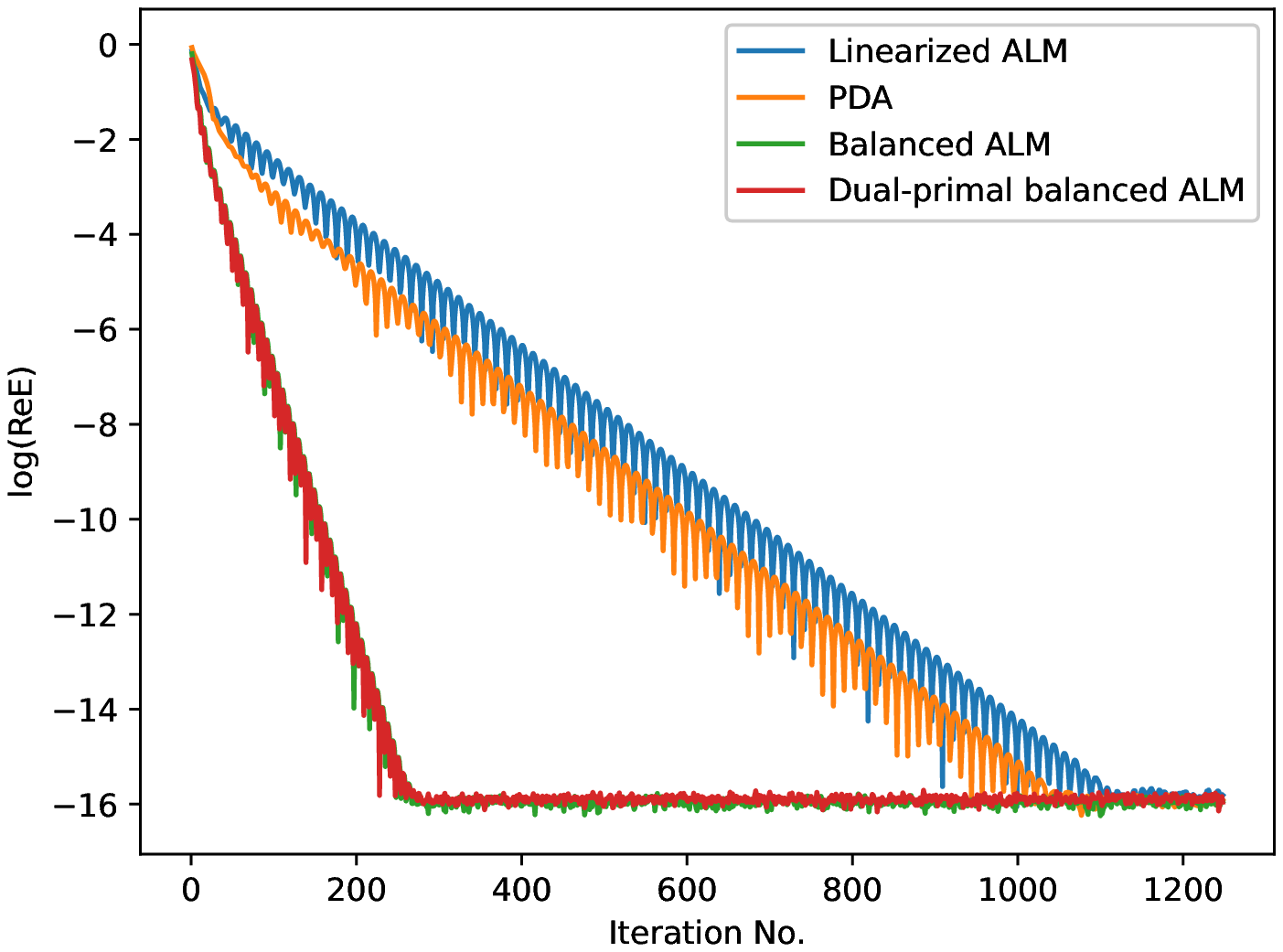}
}\hspace{-10mm}
\subfigure[$n=300$]{
\includegraphics[width=8.0cm]{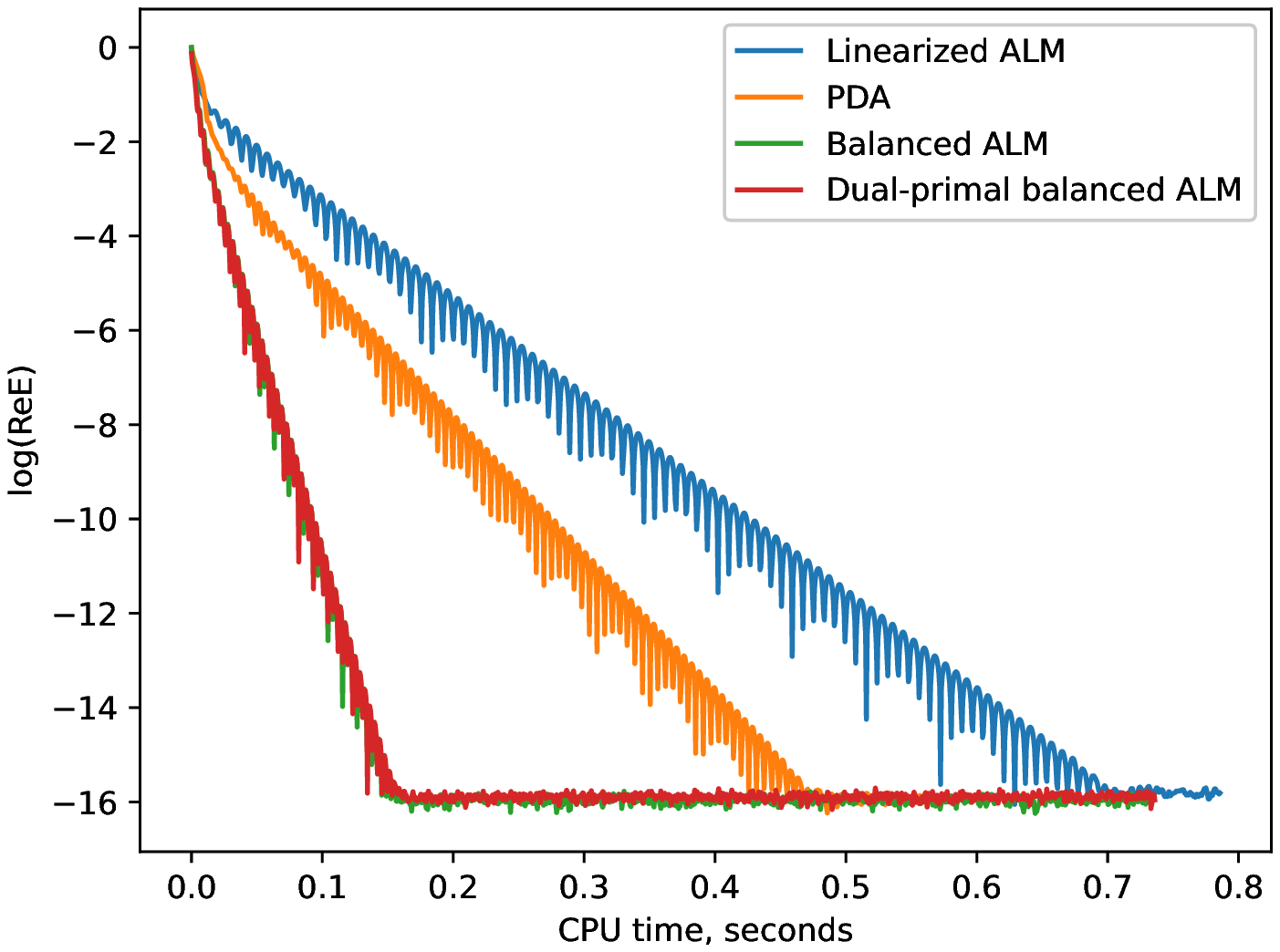}
}\\
\subfigure[$n=3000$]{
\includegraphics[width=8.0cm]{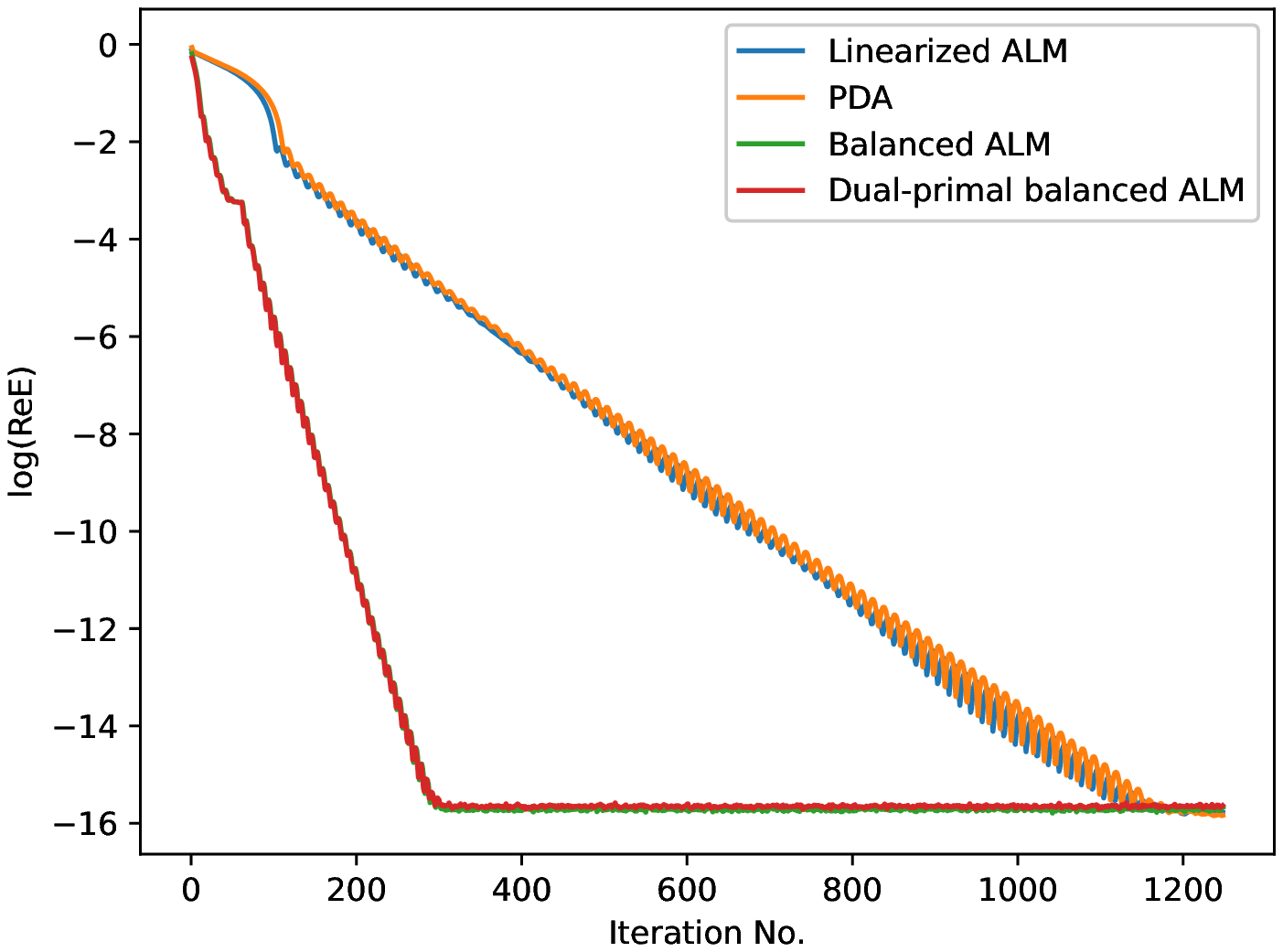}
}\hspace{-10mm}
\subfigure[$n=3000$]{
\includegraphics[width=8.0cm]{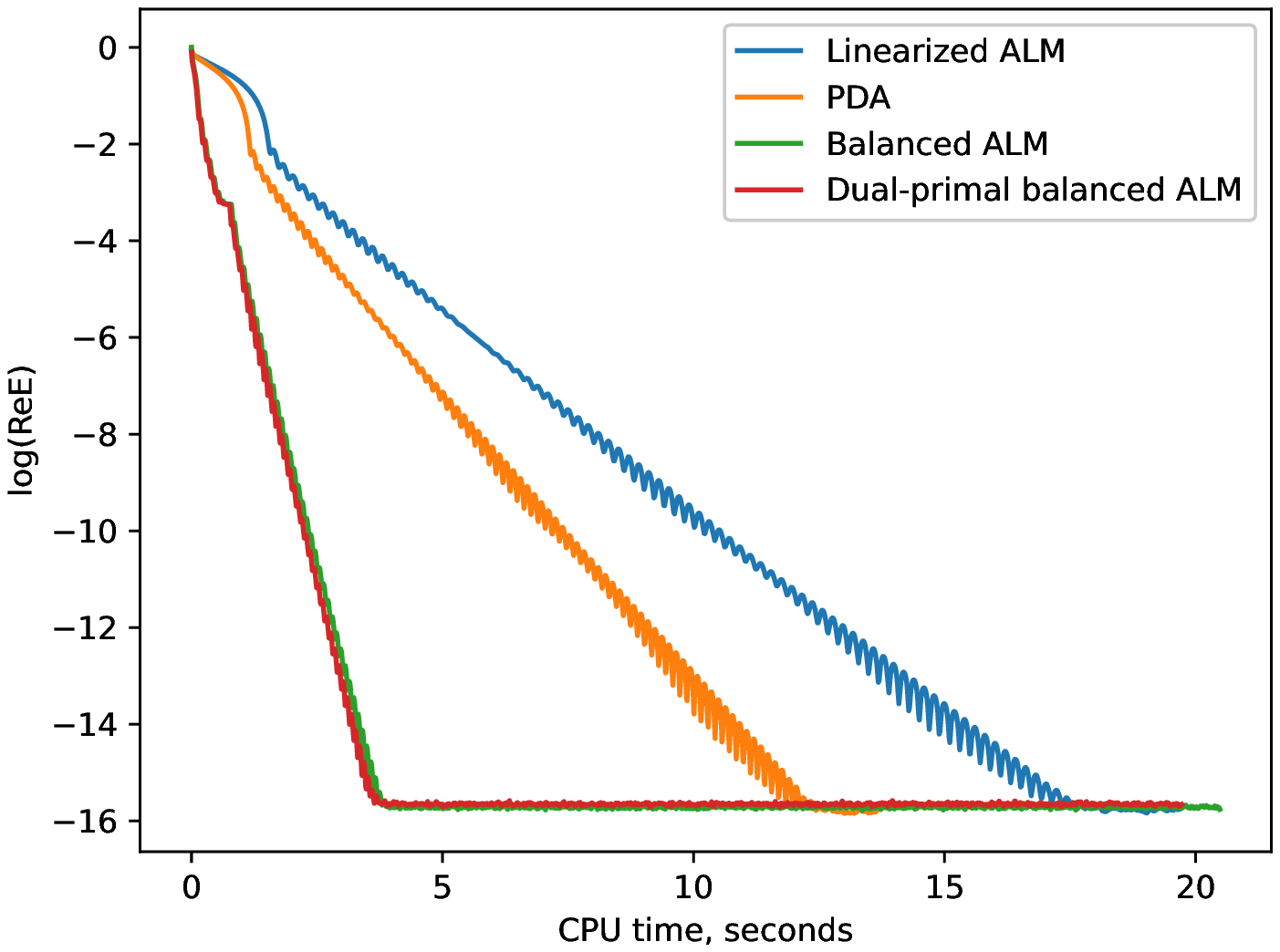}
}
\caption{Convergence curves  for the basis pursuit problem \eqref{BP} solved by the linearized ALM, the PDA, the balanced ALM and the dual-primal balanced ALM.}
\label{fig1}
\end{figure}

\section{Conclusions}\label{sec6}
\setcounter{equation}{0}
\setcounter{remark}{0}

In this short note, we present a dual-primal  balanced ALM for the canonical convex programming problem with linear equality constraints, which uses a conversely dual-primal iterative order compared with the prototype balanced ALM. It can be also generalized to tackle more general convex programming problems with both linear equality and inequality constraints. The preliminary numerical results on basis pursuit problem demonstrate that the proposed method enjoys the almost same high efficiency with the original balanced ALM. This work may significantly enhance the rich literature for the original ALM and particularly the most recent balanced ALM.

\section*{Acknowledgements}

The author is greatly indebted to Professor Bingsheng He, Nanjing University,  for numerous enlightening discussion and his helpful comments and suggestions.

\end{CJK*}
\end{document}